\documentclass[11pt]{article}
\usepackage{amsmath, amscd, amsfonts, amssymb, amsthm, tikz, times}
\usepackage[T1]{fontenc}
\newtheorem{thm}{Theorem}[section]
\newtheorem{cor}[thm]{Corollary}
\newtheorem{lem}[thm]{Lemma}
\newtheorem{prop}[thm]{Proposition}
\theoremstyle{definition}
\newtheorem{defi}[thm]{Definition}
\newtheorem*{remark}{Remark}

\def\tL{{\tilde L}}
\def\tr{\text{tr}}

\def\ed{\stackrel{d}{=}}
\def\bi{\begin{itemize}}
\def\ei{\end{itemize}}
\def\aa{\alpha}

\def\eps{\epsilon}

\def\be{\begin{equation}}
\def\ee{\end{equation}}
\def\bea{\begin{eqnarray}}
\def\eea{\end{eqnarray}}
\def\nn{\nonumber}
\def\ff{\infty}
\def\({\left(}
\def\){\right)}
\def\[{\left[}
\def\]{\right]}

\def\lb{\left|}
\def\rb{\right|}

\def\la{\langle}
\def\ra{\rangle}

\def\Var{\text{Var}}

\def\Spec{\mu}
\def\weak{\stackrel{\text{w}}{\to}}
\def\vague{\stackrel{\text{vag}}{\to}}
\def\emp{{\tiny\varnothing}}
\def\ER{Erd\H os-R\'enyi}
\def\LK{L\'evy-Khintchine}
\def\CS{Cauchy-Stieltjes}
\def\r{\mathcal{R}}
\def\cc{\mathcal{C}}

\def\aa{\mathcal{A}}
\def\ss{\mathcal{C}}
\def\hh{\mathcal{H}}
\def\Roo{{\R\backslash\{0\}}}

\def\core{\DD_{\text{fs}}}

\def\cPWIST{\cc_{G_{\infty}}}

\def\Pii{\lambda_\Pi}

\def\N{\mathbb{N}}
\def\R{\mathbb{R}}

\def\C{\mathbb{C}}

\def \P{\mathbf{P}}
\def \E{\mathbf{E}}

\def \BB{{\cal B}}

\def \DD{{\cal D}}

\def \GG{{\cal G}}

\def \II{{\cal I}}

\def \LL{{\cal L}}

\def \NN{{\cal N}}

\def \SS{{\cal S}}

\begin{document}

\title{\LK\ random matrices and the Poisson weighted infinite skeleton tree}
\author{Paul Jung\footnote{Research partially supported by NSA grant H98230-14-1-0144.}
\\[2ex]
{\normalsize Department of Mathematics, University of Alabama Birmingham}\\[1ex]
   }

\maketitle

\abstract{We study a class of Hermitian random matrices
which includes Wigner matrices, heavy-tailed
random matrices, and sparse random matrices such as adjacency matrices
of \ER\ random graphs with $p_n\sim\frac 1 n$. Our $n\times n$
random matrices have real entries which are i.i.d. up to symmetry. The distribution of entries
depends on $n$, and we require row sums to converge in
distribution; it is then well-known that the limit
distribution must be infinitely divisible.

We show that a limiting empirical spectral distribution
(LSD) exists, and
via local weak convergence of associated graphs, the LSD corresponds to the
spectral measure associated to the root of a graph which is  formed by connecting infinitely many Poisson
weighted infinite trees using a backbone structure of special edges called ``cords to infinity''. One example covered by
the results are matrices with i.i.d. entries having infinite second
moments, but normalized to be in the Gaussian domain of attraction. In this case, the limiting graph is $\N$ rooted at $1$, so the LSD is
the semi-circle law. The results also extend to
self-adjoint complex matrices and also to Wishart matrices.
}


\vspace{5mm}


\vspace{5mm}


\newpage

\tableofcontents

\section{Introduction}
This paper jointly studies the limiting
spectral distributions (LSD) for three classes of Hermitian random matrices
that have appeared in the literature. The first class of random
matrices are classic Wigner matrices introduced in the seminal
work of their namesake, \cite{wigner1955characteristic}. The
literature on this class of random matrices is overwhelmingly
abundant (see \cite{anderson2010introduction, bai2010spectral, tao2012topics}).

The second class of matrices are adjacency matrices of
Erd\H os-R\'enyi random graphs on ${n}$ vertices whose edges are
present with probability proportional to $1/{n}$. The analysis of the
LSD in the context of random matrices seems to have started in
\cite{rodgers1988density}.  These matrices are called {\it
sparse}\footnote{Here, the random number of non-zero entries in each row remains
bounded in distribution as ${n}\to\infty$. The term ``sparse''
sometimes refers to what others call dilute random matrices for which the order of
non-zero entries in each row is $\textit{o}({n})$. } {\it random matrices}, and
they can be considered a Poissonian variation of Wigner matrices.
The LSD of sparse random matrices was analyzed using the ``moment
method'' in \cite{ryan1998limit, bauer2001random, khorunzhy2004eigenvalue, zakharevich2006generalization}, and
using the ``resolvent method'' in \cite{khorunzhy2004eigenvalue}. An insightful modification of the
latter approach led to  improved results in \cite{bordenave2010resolvent}.
(see
also \cite{kuhn2008spectra} for references in the physics
literature).

Finally, the third class of random matrices are formed from properly
normalized heavy-tailed entries and, following
\cite{benarous2008spectrum}, we call them {\it heavy-tailed random
matrices}. These are also known in the physics literature as L\'evy matrices or Wigner-L\'evy matrices, and they were
introduced by Cizeau and Bouchaud in \cite{cizeau1994theory}. Later,
they were studied more rigorously in \cite{soshnikov2004poisson, benarous2008spectrum,
bordenave2011spectrum}. These matrices are {\it not} to be confused
with {\it free} L\'evy matrices \cite{benaych2005classical,
burda2007random}.

In each of the three classes of matrices above, the entries are
i.i.d. up to self-adjointness, although the distributions may
differ for different ${n}$. In order to obtain non-trivial LSDs, a
proper rescaling or change in distribution is needed as
$n\to\infty$ (such rescaling is often implicit in the
formulation). After respectively rescaling, if one sums all the
entries in a single row or column and takes ${n}\to\infty$, then one
obtains a Gaussian, Poisson, or stable distribution in each of the
respective classes. These are all examples of infinitely divisible
distributions which suggests that all three classes of matrices can many times
be thought of under this umbrella, and various papers (for example \cite[Sec. 3.1]{ryan1998limit}) have done exactly that.
More recently,
\cite{benaych2013central} establishes a functional central limit theorem for the \CS\ transforms of the LSDs
of all three classes, and \cite{male2012limiting} studies the joint LSDs of a pair of independent ensembles in
these three classes using algebraic techniques inspired by free probability.

Here, we also view these three classes as examples from this
larger class of matrix ensembles characterized by the \LK\ formula, and in particular, the matrices are viewed as (weighted) adjacency matrices. As was done in the heavy-tailed setting in \cite{bordenave2011spectrum}, our main objective is to equate the LSD of the limiting adjacency operator with the spectral measure associated to the root (or vacuum state) vector in
$L^2(V)$ where $V$ is the vertex set of the limiting graph in the sense of local weak convergence (see below).
This allows for further analysis of the LSD using the recursive structure of the limiting graph.

The ensembles we consider have
i.i.d. complex entries for each ${n}$, up to self-adjointness, with zeros on the diagonal. It is
well-known that any weak limit of row sums must be
infinitely divisible in $\C$ (viewed as $\R^2$). Actually, the ``identically
distributed'' condition may be weakened to require only that the
moduli of the entries are identically distributed. In this weakened form one still has that the sum
of the square-moduli of entries in a row, i.e., the Euclidean norm-squared of a row as a vector in $\R^{2n}$, converges in distribution to a positive  law
which is the marginal distribution of a L\'evy subordinator.

In particular,
recall (see \cite{kyprianou2006introductory} or \cite{kallenberg2002foundations}) that a probability distribution $\mu$ on $\R$ is infinitely divisible  with distribution $ID(\sigma^2,b,\Pi)$ and
 L\'evy exponent $\Psi$,
$$e^{\Psi(\theta)}:= \int_{\R} e^{i\theta {x}} \mu(d{x}) \quad\text{for } \theta\in\R,$$
if and only if there exists a triplet of characteristics
$(\sigma^2,b,\Pi)$ such that
 \be\label{psi}
\Psi(\theta):= -\frac 1 2 \theta^2\sigma^2+ i \theta b
+\int_{\R } e^{i\theta {{} x}}-1-\frac{i\theta {{} x}}{1+x^2} \Pi(d{{} x}),\ee where $\sigma^2\ge 0, b\in\R $, and
$\Pi(d{{} x})$ concentrates on $\Roo$ and satisfies
\be \label{levymeasure} \int_{\R } (1\wedge |{{} x}|^2)\,\Pi(d{{} x})<\infty. \ee
If $\mu$ concentrates on $(0,\infty)$ then the exponent corresponds to the subordinator characteristics
$(b_s,\Pi_s)$ and takes the simplified form
 \be\label{psi2}
\Psi_s(\theta):=i \theta b_s
+\int_{(0,\infty)} (e^{i\theta {{} x}}-1)\, \Pi_s(d{{} x}),\ee where
$\Pi_s(d{{} x})$ also concentrates on $(0,\infty)$, but instead of \eqref{levymeasure}, it satisfies \be \nn\label{levy
measure subordinator} \int_{(0,\infty) } (1\wedge {x})\,\Pi_s(d{{} x})<\infty. \ee
Here, the $s$ subscript indicates the {\it subordinator} form of the L\'evy exponent.  

We say a sequence of  ${n}\times {n}$ random matrices
$\(\cc_{n}\)_{n\in\N}$ is a {\bf \LK\ random matrix ensemble} with
characteristics $(\sigma^2,0,\Pi)$ if for each $n$, the moduli of entries
$\cc_n(j,k)=\bar\cc_n(k,j), j\neq k$ are i.i.d. (up to
self-adjointness, with zeros on the diagonal) and the \be
\label{LKcondition} \text{weak limit }\
\lim_{n\to\infty}\sum_{k=1}^{n} \pm|\cc_n(1,k)|  \ \text{ is infinitely divisible with characteristics
}\
(\sigma^2,0,\Pi),\ee where the signs $\pm$ are independent Rademacher random variables  (independent also from $\cc_n$). This implies that $\Pi$ is a
symmetric measure. It is not hard to see that \eqref{LKcondition} is true if and only if
$$\lim_{n\to\infty}\sum_{k=1}^{n} \cc_n(1,k)$$ is infinitely divisible with some other characteristics $(\sigma^2,\tilde{b},\tilde{\Pi})$ where $\sigma^2$ remains unchanged, but $\tilde{b}$ may be nonzero and $\tilde{\Pi}$ is not in general symmetric.
An equivalent form of the above is that the
\be
\label{LKcondition2} \text{weak limit }\
\lim_{n\to\infty}\sum_{k=1}^{n} |\cc_n(1,k)|^2  \ \text{ is
infinitely divisible with subordinator characteristics }\
(\sigma^2,\hat{\Pi}_s)\ee
where $\hat{\Pi}_s$ can be  easily found in terms of $\Pi$ (see \cite[Ch. 15]{kallenberg2002foundations}).
Note that in this form $\sigma^2$ plays the role of the drift coeffecient $b_s$.
We note that by standard arguments, one could set
the diagonal elements to any real number which converges to $0$ fast
enough, and this would not affect the LSD. For the sake of simplicity, we will always
set diagonal entries to zero.

In the context of L\'evy processes, the three components of the
triplet $(\sigma^2,b,\Pi)$ correspond to a Brownian component, a
drift component, and a jump component (with possibly additional
``compensating drift''), respectively. We will see in our context
that $\sigma^2$ corresponds to a Wigner component, the drift
component is inconsequential since by using the random signs it becomes 0 (cf. \cite[Remark
1.9]{benarous2008spectrum}), and the L\'evy measure $\Pi$
generalizes both heavy-tailed and sparse random matrices.

Let $\(\aa_{n}\)_{n\in\N}$ denote an ensemble which satisfies the
above conditions except it does not require the condition of
self-adjointness, $\cc_n(j,k)=\bar\cc_n(k,j)$.  We call this
a non-Hermitian \LK\ random matrix ensemble.
Using a standard bipartization/Hermitization method\footnote{This
method has  appeared in the physics literature (see
\cite{feinberg1997non} and the references therein) and is discussed in the texts
\cite{anderson2010introduction, tao2012topics}.}, our
results extend to the LSD of Wishart matrices $\(\aa^* \aa_n\)_{n\in\N}$ or
equivalently to the limiting empirical singular value distribution
for $\(\aa_n\)_{n\in\N}$.

\subsection{Main results}

For a given \LK\ ensemble, let $\{\lambda_j\}_{j=1}^{n}$ denote the
eigenvalues of the ${n}$th matrix in the sequence. The empirical
spectral distribution (ESD) is defined as \be \label{esd}
\Spec_{\cc_n}:= \frac 1 {n} \sum_{j=1}^{n} \delta_{\lambda_j}. \ee

\begin{thm}[Existence of the LSD]\label{thm:main}
For any \LK\ random matrix ensemble $\(\cc_{n}\)_{n\in\N}$ with characteristics
$(\sigma^2,0,\Pi)$ (or alternatively with subordinator characteristics $(\sigma^2,\hat{\Pi}_s)$), there exists a symmetric nonrandom probability measure
$\Spec_{\cc_\infty}$ to which the ESDs $(\Spec_{\cc_n})_{n\in\N}$ weakly
converge, almost surely, as ${n}\to\infty$. In other words,
\be \P\(\lim_{{n}\to\infty}
\langle \Spec_{\cc_n}, f\rangle = \langle \Spec_{\cc_\infty},
f\rangle \text{ for all bounded continuous }f\) = 1. \ee
Moreover, the limiting measure $\Spec_{\cc_\infty}$ has bounded support if and only if $\Pi$ is trivial.
\end{thm}

An extension of the above result to the singular values of
$\(\aa_n-zI_n\)_{n\in\N}$ in the spirit of
\cite{dozier2007empirical}  follows by way of  Theorem 2.1 in \cite{bordenave2011nonHermitian}
(see also \cite{feinberg1997non}). This gives us the following corollary.
\begin{cor}[LSD for Wishart ensembles]\label{cor2}
Suppose $\(\aa_n\)_{n\in\N}$ is a non-Hermitian \LK\
ensemble  with characteristics $(\sigma^2,0,\Pi)$. The LSD,
$\nu_{\infty}$, of the Wishart ensemble $\(\aa^*\aa_n\)_{n\in\N}$
exists and is given by
$$\nu_{\infty}(B)=\Spec_{\cc_\infty}\{x:x^2\in B\}$$
where $\Spec_{\cc_\infty}$ is the LSD from Theorem \ref{thm:main} for the Hermitian ensemble with the same  characteristics.
\end{cor}

In the case where $\Pi$ has exponential moments, an extension of the standard moment method is enough
to handle the proof of Theorem \ref{thm:main}, and in Section \ref{sec:moment method} we do just that under the slightly stronger assumption that $\Pi$ has bounded support.
When $\Pi$ has some moments which are infinite and $\sigma=0$,
the proof follows by generalizing insightful
{\it local weak convergence} arguments of \cite{bordenave2010resolvent,
bordenave2011spectrum} (see Section \ref{sec:nowigner}). To extend this to the general case, we
combine the local weak convergence arguments with  a generalized moment method, and tail truncation arguments.

As a by-product of local weak convergence, one can view the LSD of the random matrix ensembles as the spectral measure
of a weighted adjacency operator, at the root vector, of some new infinite graph. For ensembles with characteristics $(0,0,\Pi)$, this
idea is again a generalization of arguments in \cite{bordenave2011spectrum}. However, when $\sigma>0$ a non-trivial generalization of Aldous' Poisson weighted
infinite tree which we call a {\bf Poisson weighted infinite
skeleton tree (PWIST)} is required.

The idea of local weak convergence was introduced by Benjamini and Schramm \cite{benjamini2001recurrence} and further developed
by Aldous and Steele \cite{aldous2004objective}. Aldous and Steele describe the technique as finding ``a new, infinite, probabilistic object whose
local properties inform us about the limiting properties of a
sequence of finite problems.'' When the limiting object has a tree structure,
local weak convergence provides a general framework to make the {\it cavity method}
in physics rigorous.
In our context, the cavity method was used in \cite{cizeau1994theory} and our new infinite
object (with a tree structure) generalizes Aldous' { Poisson infinite weighted trees
(PWIT)} by adding to it ``cords'' of infinite length which connect
to independent copies of other PWITs. These cords form a {\it
backbone structure} for a collective object which we refer to as a
PWIST.

Let us first recall the definition of the PWIT$(\Pii)$. Start with a
single root vertex $\emp$ with an infinite number of (first
generation) children indexed by $\N$. The weight  on the edge to the
$k$th child  is the $k$th arrival (ordered by absolute value) of a Poisson process on $\Roo$
with some intensity $\lambda$. In our situation the intensity $\Pii$
is derived from the measure $\Pi$ on $\Roo$ by inverting:
\be \label{pi} \Pii\{x:1/x\in B\}:=\Pi(B) \ee For example, if
$\Pi(dx)$ is equivalent to Lebesgue measure with density $f_\Pi(x)dx$ then
$\Pii(dx)$ is also equivalent to Lebesgue measure  with density
$x^{-2}f_\Pi(1/x)dx$ where $x^{-2}$ is the change-of-measure factor.

If $G$ has a root at $\emp$ we write $G[\emp]$ for the rooted graph
with (random) weights assigned to each edge. Slightly
abusing notation, we denote the subgraph of a PWIT($\Pii$) formed by
the root $\emp$, its children, and the weighted edges in between, by
$\N[\emp]$.

\includegraphics[scale=.4]{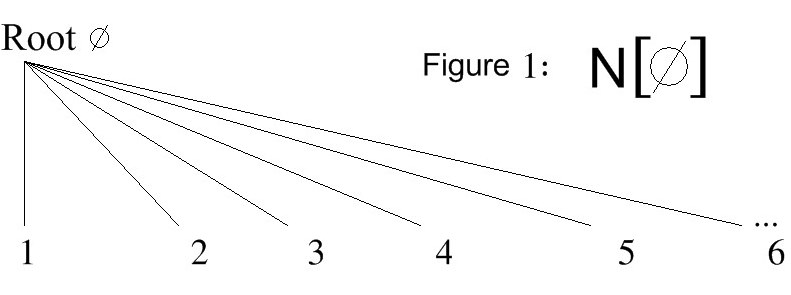}

We continue now with other generations. Every vertex $v$ in
generation (or depth) $g\ge 1$ has edges to an infinite number of children indexed
by $\N$ forming the subgraph $\N[v]$, with weights assigned by repeating the procedure for the weights in the first generation (for $\N[\emp]$), namely according
to the points of an {\it independent} Poisson random measure with
intensity $\Pii(dx)$. Therefore each $\N[v]$ is an i.i.d. copy of $\N[\emp]$.  The union of the children vertices of $\N[v]$ (in other words, not including $v$ itself) over all $v$ in some generation $g-1$ is denoted $\N^g$. Thus the total vertex set is \be
\label{vertices PWIT} \N^F:=\bigcup_{g\ge 0} \N^g \ee where
$\N^0=\emp$.

\vspace{3mm}
\includegraphics[scale=.45]{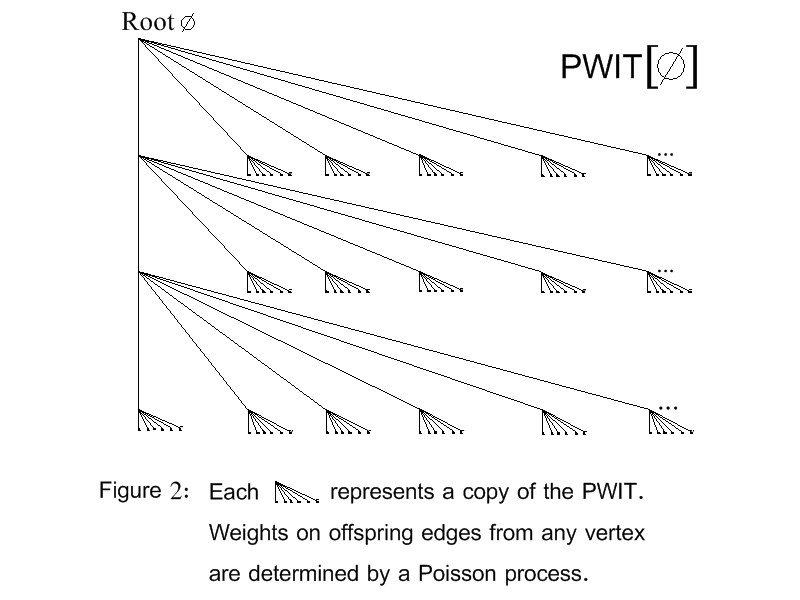}

The PWIST will depend on both characteristics $\sigma^2$ and $\Pi$
(via $\Pii$). To construct a PWIST$(\sigma,\Pii)$, we start with a
single PWIT$(\Pii)[\emp]$ rooted at $\emp$ and, for each vertex $v$
of PWIT$(\Pii)[\emp]$, we create a new vertex $\infty_v$ which is
the root of a new independent PWIT$(\Pii)[\infty_v]$. We draw an
edge from $v$ to $\infty_v$ for each $v$ and assign this edge  a
nonrandom weight of \be \label{distancea} 1/\sigma \in(0,\infty].
\ee Next, we create a new independent PWIT$(\Pii)[\infty_u]$ for
each vertex $u$ of each PWIT$(\Pii)[\infty_v]$, and draw an edge
with weight $1/\sigma $ between $u$ and $\infty_u$.  We continue
this procedure
ad infinitum.
If we also identify  $\infty_v$ with the integer $0$ so that by
concatenation, $\infty_v$ is written $v0$, then  we may write the
vertex set of a PWIST$(\sigma,\Pii)$ as
 \be\label{nof}
\N_0^F:=\bigcup_{g\ge 0} \N_0^g
 \ee
 where $\N_0=\N\cup\{0\}$ and by concatenation we write $v=v_1v_2\cdots v_g\in\N_0^g$.
As can be seen in the figure below, edges with the weight $1/\sigma$ connect infinitely many PWITs with a
 backbone structure in order to form a PWIST.

\includegraphics[scale=.42]{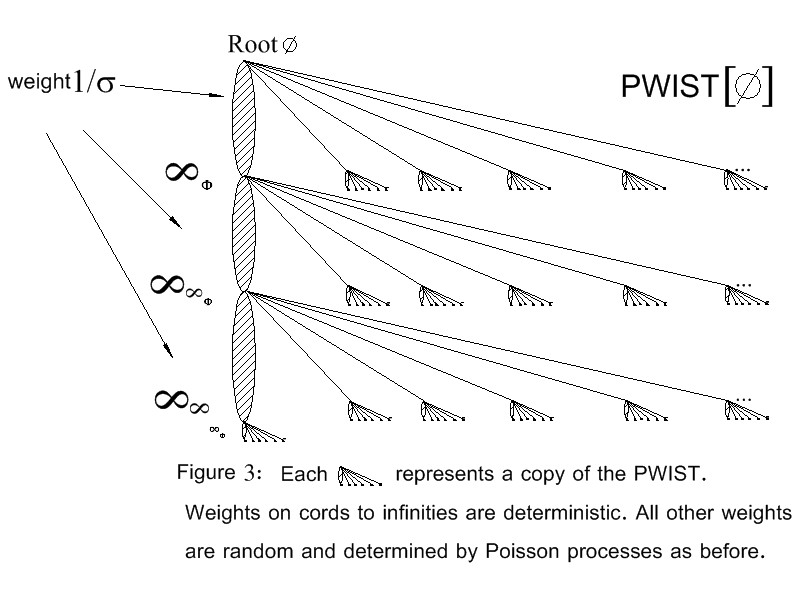}

Our next theorem justifies the choice \eqref{distancea} for the
weight on the edge between $v$ and $\infty_v$. Let us however attempt a
brief heuristic explanation as to why this is the correct weight to
assign to this edge. First of all, identify each weight with its
absolute value so that all weights are thought of as {\it
nonnegative conductances}. Now, if $\sigma=0$, then the connected
graph containing the root $\emp$ is simply a PWIT$(\Pii)$ with the
weights on edges representing nonnegative conductances. If
$\sigma>0$, we use the interpretation that $v$ and $\infty_v$ are
infinitely far apart, but also that there are infinitely many
parallel edges (or a multi-edge) between $v$ and $\infty_v$. Since distance is
equivalent to resistance on electrical networks and resistance is
the reciprocal of conductance, the conductance of each parallel edge
is zero; however, their collective {\it effective} conductance is greater than $0$, and in particular is of order $\sigma$. We can thus identify the multiple parallel edges with a {\it single} edge between $v$ and $\infty_v$ called
a {\bf cord to infinity} with effective resistance $1/\sigma $.

Let us now consider a random weighted adjacency matrix $\cc_{G_{n}}$
 associated to a complete {\bf rooted geometric graph} (see Section \ref{sec:nowigner} for definitions)
$G_{n}=G_{n}[\emp]=(V_n,E_n,\r_n)$ where $V_n=\{1,\ldots,n\}$ and
$\r_n$ are the (possibly signed) random weights/lengths/resistances of the edges
$E_n$. We refer to such a real-valued matrix as a {\it random conductance matrix}
with entries given simply by the reciprocals of the
signed resistances:
\begin{equation}
\label{eq10}
\cc_{G_n}(j,k):=\frac{1}{\r_n(j,k)} .
\end{equation}
When a sequence of random
conductance matrices satisfies
\eqref{LKcondition} or \eqref{LKcondition2}, it forms a \LK\ random matrix ensemble.

This notion generalizes to a {\bf random conductance operator} on
$L^2(G_\infty)\equiv L^2(V_{\infty})$ for an infinite weighted graph
$G_\infty=(V_\infty,E_\infty,\r_\infty)$. Let the core $\core\subset
L^2(V_\infty)$ be the set of vectors with finite support, i.e., all
finite linear combinations of the basis vectors $e_v$ which are $1$ at $v$ and $0$ elsewhere. We
consider the operator on $\core$ which is defined by \be \label{aPWIST}
\cc_{G_\infty}(u,v) =\la e_u, \cc_{G_\infty}e_v \ra:=
\begin{cases}
& 1/\r_\infty(u,v)\quad \text{if } u\sim v\\
& 0 \quad\quad \quad \quad\quad\  \text{otherwise}.
\end{cases}
\ee
This operator is closable as a graph in $L^2(V_{\infty})\times L^2(V_{\infty})$ since it is symmetric, i.e., Hermitian and densely defined \cite[Thm 5.4]{weidmann1980linear}. Abusing notation we also denote
its unique closure by $\cc_{G_\infty}$. In particular, we will see that the closure is self-adjoint.
In the case where $G_\infty$ is a $\text{PWIST}(\sigma,\Pii)$,
by \eqref{pi}, the conductances are given by the points of a Poisson
random measure with symmetric intensity $\Pi(dx)$ on $\Roo$.

Now, recall \cite[Sec. VII.2 and VIII.3]{reed1980methods} that the spectral measure $\mu_\varphi$ of a self-adjoint operator $\cc$  associated to the vector $\varphi$
is defined by the relation
$$\langle \varphi, f(\cc)\varphi\rangle =:\int_\R f(x) \mu_\varphi(dx), \quad \text{for bounded continuous } f.$$

\begin{thm}[LSD as the root spectral measure of a limiting operator]\label{thm:main2}
For any \LK\ ensemble $\(\cc_{n}\)_{n\in\N}$ with characteristics
$(\sigma^2,0,\Pi)$, the limiting spectral distribution
$\Spec_{\cc_\infty}$ of Theorem \ref{thm:main} is the expected spectral measure, at the root vector $e_\emp$, 
of a self-adjoint random conductance operator $\cPWIST$ on
$L^2(\N_0^F)$ where $G_\infty$ is a PWIST($\sigma,\Pii$).
\end{thm}

\noindent {\bf Remarks:}
\begin{enumerate}
\item The symmetry of the measure $\Spec_{\cc_\infty}$ is now easy to see, since every PWIST($\sigma,\Pii$) is a tree, and thus the odd moments of
$\Spec_{\cc_\infty}$ vanish.
\item The above matrix ensembles can be decomposed, by the L\'evy-It\=o decomposition, into $(\cc_n)_{n\in\N}$ and $(\cc'_n)_{n\in\N}$ which are independent with characteristics $(0,0,\Pi)$ and $(\sigma^2,0,0)$. The sequence $(\cc_n+\cc'_n)_{n\in\N}$ then
has characteristics $(\sigma^2,0,\Pi)$. One approach is to try and generalize Voiculescu's asymptotic freeness theorem to establish the above result, however, we have been unable to do so due to the randomness of the PWIT associated to the L\'evy measure $\Pi$ (if the graphs were deterministic, one could use the approach of \cite{accardi2007decompositions}).
\end{enumerate}

The following result is an application of the resolvent identity,
and it may be used in conjunction with Theorem \ref{thm:main2} to further analyze $\Spec_{\cc_\infty}$. It can be viewed as an operator version of the Schur complement formula.

\begin{prop}[Recursive distributional equation]\label{cor1}
Suppose $G_\infty$ is a PWIST($\sigma,\Pii$). For all $z\in\C_+$ the
random variable
$$R_{\emp\emp}(z):=\la e_{\emp},(\cPWIST-zI)^{-1}e_{\emp} \ra$$
satisfies $R_{\emp\emp}(-\bar z)=-\bar R_{\emp\emp}(z)$ and the recursive distributional equation (RDE)
\be
\label{eq:rde} R_{\emp\emp}(z) \ed -\(z+ \sigma^2 R_{00}(z) + \sum_{k\in\N}|\cc(k)|^2 R_{kk}(z) \)^{-1} \ee
where for all  $k\ge 0$, $R_{kk}$ has the same distribution as $R_{\emp\emp}$
and $\{\cc(k)\}_{k\in\N}$ are the points of an independent Poisson random measure with intensity  $\Pi(dx)$ on $\Roo$.
\end{prop}

\noindent {\bf Remarks:}
\begin{enumerate}
\item One can extend the proposition to the Wishart ensembles (as in Corollary \ref{cor2}) using Lemma 2.5 in \cite{bordenave2011nonHermitian}.
\item
For an example of how the above proposition may be used, consider Wigner matrices with i.i.d. entries with possibly infinite second
moments, but normalized to be in the Gaussian domain of attraction. In this case, the L\'evy measure $\Pi$ is trivial and the PWIST($\sigma,0$) is just $\N$ rooted at $1$.

\includegraphics[scale=.4]{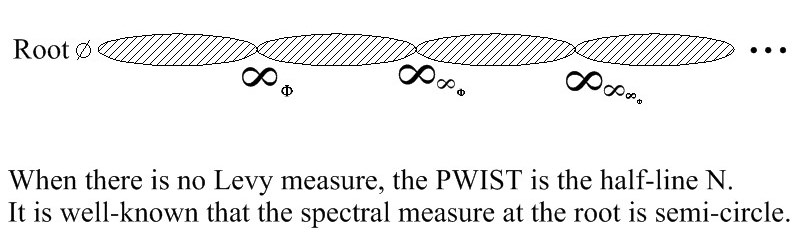}

\noindent Since the edge-weights of the limiting graph are nonrandom, a simple argument shows (see Eq. \eqref{eq:resolvent identity} below) that
the resulting recursive equation is the \CS\ transform (see \eqref{def:Stieltjes}) of Wigner's semi-circle law: $$R_{\emp\emp}(z)=S_{\mu_{sc}}(z) = -\(z+ \sigma^2 S_{\mu_{sc}}(z)\)^{-1}.$$
\end{enumerate}

The rest of the paper is organized as follows. In the next section, we introduce a replacement procedure which creates a new sequence of
matrices by modifying a given \LK\ ensemble.
This modification replaces complex values with real values and also embodies our notion of ``cords to infinity''. It is the key procedure which allows us to generalize
PWITs to PWISTs.
In Section \ref{sec:moment method}, the moment method is used to prove a weak version of Theorem \ref{thm:main}
in the case that the L\'evy measure $\Pi$ has bounded support.  The main point
of Section \ref{sec:moment method}, however, is to show that the limiting root spectral measure of a \LK\ ensemble is invariant under the replacement procedure
of Section \ref{replacement} (in preparation for proofs of the main results). In Section \ref{sec:nowigner}, we precisely define local weak convergence and
present an adaptation of the arguments of \cite{bordenave2011spectrum}.  In particular, we show that the local weak convergence argument proves Theorem \ref{thm:main2} for real \LK\ ensembles with
$\sigma=0$.  Finally, in Section \ref{sec:proofs}, we combine the arguments of Sections \ref{sec:moment method} and \ref{sec:nowigner} to
prove the main results in the general case.  In the appendix we gather some known results which are needed along the way.

\paragraph{Acknowledgments} This project arose out of several discussions with Shannon Starr.
We thank him for being a sounding board and for his many helpful discussions,  insights, and suggestions.

\section{A replacement procedure for cords to infinity}\label{replacement}

In this section, we define an important sequence of {\it modified} matrices $\(\cc^{\sigma}_n\)_{n\in\N}$ which play a key role in the proofs of the main results.
In particular, these matrices are
modifications of  a \LK\ ensemble $\(\cc_n\)_{n\in\N}$ under a
certain replacement procedure which we describe below.

For $h>0$, by  \eqref{LKcondition} and Proposition \ref{Kall
cor15.16} we have that as $n\to\infty$,
$$\sum_{k=1}^n \pm|\cc_n(1,k)|1_{\{|\cc_n(1,k)|\le h\}} $$ converges in distribution to  $ID(\sigma^2_h,0,\Pi_h)$
where the $\pm$ signs are chosen using independent Rademacher variables (independent also from $\cc_n$), and
\bea\nn
\sigma^2_h&:=&\sigma^2+\int_{|{{} x}|\le h}{{{} x}^2}\, \Pi(d{{} x})\quad\text{ and}\\
\nn \Pi_h(dx)&:=&1_{(-\infty,-h]\cup[h,\infty)}(x)\Pi(dx).
\eea
By a diagonalization argument, we may choose a sequence of positive numbers $h_n\to 0$ such that
we get the following weak convergence to a Gaussian:
$$\sum_{k=1}^n \pm|\cc_n(1,k)|1_{\{|\cc_n(1,k)|\le h_n\}} \Rightarrow \NN(0,\sigma^2).$$
In particular, as $h_n\to 0$
 \bea\label{eq:kal variance}
&&\lim_{n\to\infty}\sum_{k=2}^n \E\(|\cc_{n}(1,k)|^2 1_{\{|\cc_n(1,k)|\le h_n\}}\)\nn\\
&=&\lim_{n\to\infty} n\E\(|\cc_{n}(1,2)|^21_{\{|\cc_n(1,2)|\le h_n\}}\)=\sigma^2. \eea

Our replacement procedure is as follows. For all entries such that
$|\cc_n(j,k)|> h_n$ as well as for all diagonal entries
$\cc_n(j,j)$, we set $\cc^{\sigma}_n(j,k):=\pm|\cc_n(j,k)|$ where the signs $\pm$ are given by independent Rademacher variables on the upper triangle, and determined
on the lower triangle to preserve self-adjointness.
However, the  entries in positions $(j,k)$, $j\neq k$ in the matrix
$\cc^{\sigma}_n$,  which satisfy the condition $|\cc_n(j,k)|\le h_n$,
will remain blank for
now and will be assigned values that are either $0$ or $\sigma$.

We next describe how to fill in blank entries. We first need to determine the order of the rows (and columns to preserve self-adjointness) by which we fill in the blanks.
Recall that
$\cc_n$ determines a  geometric graph, rooted at $1$, with edge-weights
given by $1/\cc_n(j,k)$ as in \eqref{eq10}.
Let $\alpha$ be
the permutation of $\{1,\ldots,n\}$ such that $\alpha(i)$
is the $i$th closest vertex from the root $1$ using the distance
\eqref{distance}. If $j$ and $k$ are at equal distance from the root $1$, we break ties by deeming $j$ ``closer'' to the root whenever $j<k$.
We now fill in blank entries according to the order determined by the (random) permutation $\alpha$. For instance, we fill in blanks in row $1$ first since $\alpha(1)=1$
(the root is always closest to itself). Next we fill in blank entries in the row $\alpha(2)$, then row $\alpha(3)$, etc.

The procedure for filling in blank entries in  row $j=\alpha(i)$ is as follows, starting  with row $1=\alpha(1)$. Out of all $k$ satisfying
\be\label{eqk} |\cc_n(j,k)|\le h_n,  \quad k\neq j\ee
choose one uniformly at random and set this entry,  in
$\cc^{\sigma}_n$, to
$\sigma$.
Set other blank entries in  row $j$, satisfying \eqref{eqk}, to zero
 in the matrix $\cc^{\sigma}_n$.
 This completes the
filling of row $j$ of $\cc^{\sigma}_n$, and we use
the symmetry condition
$\cc^{\sigma}_n(j,k)=\cc^{\sigma}_n(k,j)$, to fill in blank entries
in the column $j$.

When row and column $j=\alpha(i)$ are completely filled, we
repeat the procedure on row and column $\alpha(i+1)$.
We
continue the replacement procedure described in the previous paragraph until all blank
entries have been filled, then we say call $\(\cc^{\sigma}_n\)_{n\in\N}$ the  {\bf modified sequence of matrices}.

\section{The moment method}\label{sec:moment method}
In this section, we use the moment method to prove  a {\it convergence in expectation}\footnote{See \cite[Remark 2.4.1]{tao2012topics} for a definition and short discussion of this type of convergence.} version of Theorem \ref{thm:main} in the case where
there exists an almost sure bound $0<\tau <\infty$ on the entries of the \LK\ ensemble $\(\cc_{n}\)_{n\in\N}$, 
\be\label{finite nu moments}
|\cc_n(1,2)|\le\tau\quad \text{for all } n.
\ee
In particular, using the associated Poisson approximation for the distribution of $\cc_n(1,2)$ (see \cite[Cor. 15.16]{kallenberg2002foundations}) one sees that $\Pi$ must be supported on $[-\tau,\tau]$.

Let $$M_p(\mu):= \int_\R x^p \mu(dx)$$
be the
$p$th moment of the measure $\mu$.
The moment method
in this section consists of showing
\be\label{moments}
\lim_{n\to\infty} M_p(\E\Spec_{\cc_n}) = M_p(\E\Spec_{\cc_\infty}),\quad\text{for all }p\in\N,
\ee
and then verifying that the moments  $M_p(\E\Spec_{\cc_\infty})$ determine $\E\Spec_{\cc_\infty}$.
However, the main result of this section is the following important consequence of such a verification.

\begin{prop}[Invariance of expected LSD under replacement procedure]\label{lemma:same}
If the expected LSD for a \LK\ ensemble $\(\cc_n\)_{n\in\N}$ exists and is determined by its moments, then it is
equal to the limiting expected spectral measure associated to $e_1$ (the first vector of the standard basis)
for any modified sequence $\(\cc_n^{\sigma}\)_{n\in\N}$. 
\end{prop}

\begin{proof}

A standard argument (see \cite[Ch. 2]{anderson2010introduction} or \cite[Sec. 2.3.4]{tao2012topics} for details) shows that the  $p$-moments
are given by
\bea\label{m2p}
M_{p}(\E\Spec_{\cc_n})&=&\E\frac{1}{n} \tr(\cc_n^{p})\nn\\
&=&\sum_{j_2,\ldots,j_{p}=1}^n\E(\cc_n(1,j_2)\cc_n(j_2,j_3)\cdots \cc_n(j_{p},1))
\eea
where we have set $j_1=1$ by exchangeability. The ordered listings of subscript pairs $$\bigl((1,j_2)(j_2,j_3),\ldots,(j_p,1)\bigr)_{j_2,\ldots, j_p=1}^n$$ are viewed as distinct  paths  of length $p$ which start and end at $1$ in the complete graph on $\{1,\ldots,n\}$, with edges having orientations,
and with the possibility that edges are crossed multiple times. These paths are called {\it cycles rooted at $1$}.


We now make some preliminary observations in order to rewrite \eqref{m2p} as \eqref{last moment eq}.  The expression of the $p$th moment in \eqref{last moment eq} below allows us to then prove the result.

By Proposition 4 in \cite{zakharevich2006generalization}, in the limit as $n\to\infty$, the only cycles
that {\it contribute to the limiting sum}  on the right-side of \eqref{m2p}  are ``trees'' in the following sense. For a given {\it contributing} term, if the oriented edge $(j_k,j_{k+1})$ is crossed $q=q(k)$ times, then
it must also be crossed $q$ times in the opposite orientation. Thus, for each $k$ there is a corresponding $k'\neq k$ such that
\be\label{distinct ells}
\cc_n(j_k,j_{k+1})=\overline{\cc_n(j_{k'},j_{k'+1})}, \quad j_k=j_{k'+1},\ j_{k+1}=j_{k'}.
\ee
Moreover, the partition of $\{1,\ldots,p\}$ which pairs each $k$ with its corresponding $k'$ must be a {\it non-crossing pair partition}  (see \cite{nica2006lectures} for details).
In particular, $p$ must be even in order to have a non-trivial moment.

If $\cc_n(j_k,j_{k+1})$ appears $q=q(k)$ distinct times in a given term, then its conjugate (or reversed edge from $j_{k'}$ to $j_{k'+1}$) also appears $q=q(k')$ distinct times.
Using independence and exchangeability,  each term of the sum in \eqref{m2p} takes the form
\be\label{qi's}
\E|\cc_n(1,2)|^{2q_1}\E|\cc_n(1,2)|^{2q_2}\cdots \E|\cc_n(1,2)|^{2q_\ell}
\ee
where $2q_1+\cdots+2q_\ell=p$.

Fix the value of $j_2$ and consider a cycle rooted at $1$ corresponding to a term in the sum \eqref{m2p} such that $(1,j_2)$ is crossed $q=q(1)$ times in each direction for a total of $2q$ times. Removing
these $2q$ edges from our cycle leaves us with several sub-cycles.  These sub-cycles can be permuted and then concatenated to form
two sub-cycles $L$ and $\tL$ rooted at $L_1:=j_2$ and $\tL_1:=1$ which avoid the edges $(1,j_2)$ and $(j_2,1)$ (one or both of the cycles may be trivial).

\includegraphics[scale=.3]{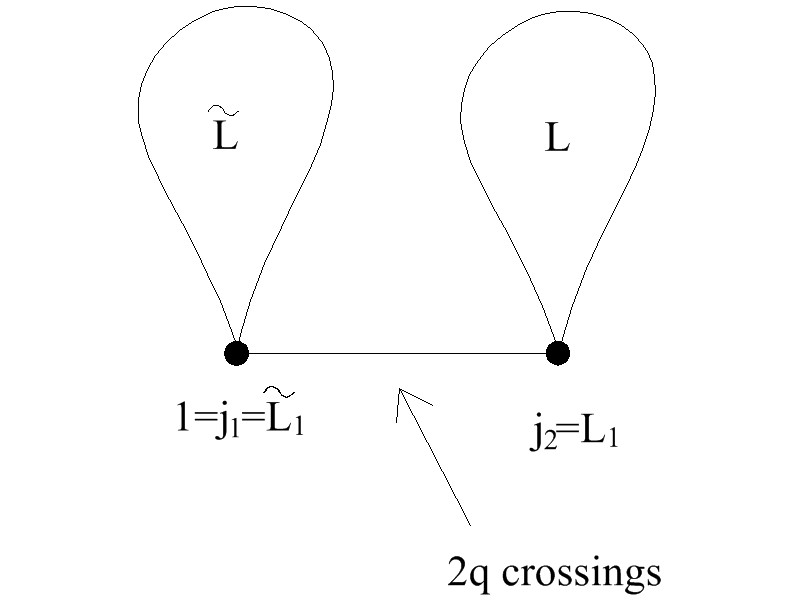}

Write $\mathbb{L}(j_2,q)$ for the set of all pairs of cycles $(L,\tL)$ which are possible, where in particular, different permutations/concatenations leading to the same $L$ or $\tL$ are each listed separately in $\mathbb{L}(j_2,q)$, i.e., $L$ and $\tL$ remember their original sub-cycle structure.
Also, let $s,\tilde s$ be the lengths of $L,\tL$ so that $s+\tilde s={p-2q}$, and write
 $L\equiv ((L_1, L_2), \ldots, (L_{s}, {L_{1}}))$ and similarly for $\tL$.
Discarding some terms which do not contribute to the limiting sum, we have that
 \eqref{m2p} can be rewritten as
\bea\label{condition on hn}\nn
&&\sum_{q=1}^{p/2}\sum_{j_2=2}^n\sum_{(L,\tL)\in\mathbb{L}(j_2,q)}\E|\cc_n(1,j_2)|^{2q}\E(\cc_n(L_1, L_2)\cdots \cc_n(L_{s},L_1)\cc_n(\tL_1,\tL_2)\cdots\cc_n(\tL_{\tilde s},\tL_1))\\
\nn&=&\sum_{q=1}^{p/2}\sum_{j_2=2}^n\{ \E\[ |\cc_n(1,j_2)|^{2q}(1_{\{|\cc_n(1,j_{2})|\le h_n\}}+1_{\{|\cc_n(1,j_{2})|>h_n\}})\]\times \\
&&\sum_{L,\tL\in\mathbb{L}(j_2,q)}\E\(\cc_n(L_1, L_2)\cdots \cc_n(L_{s},L_1)\cc_n(\tL_1,\tL_2)\cdots\cc_n(\tL_{\tilde s},\tL_1))\)\}.
\eea

Recall from \eqref{eq:kal variance} that for $\epsilon>0$, we may find $N$ such that $n\ge N$ implies
\be\label{max2}
n\E\(|\cc_n(1,2)|^{2}1_{\{|\cc_n(1,2)|\le h_n\}}\) \le \sigma^2+\epsilon.
\ee
which in turn implies
\be\label{max2q}
n\E\(|\cc_n(1,2)|^{2q}1_{\{|\cc_n(1,2)|\le h_n\}}\) \le h_{n}^{2q-2}(\sigma^2+\epsilon).
\ee
To see this, note that a distribution satisfying \eqref{max2} with maximum $2q$th moment is given by $\cc_n(1,2)= \pm h_n$ with probability $\frac{\sigma^2+\epsilon}{nh_n^2}$ and $\cc_n(1,2)= 0$
otherwise. Since $h_n\to 0$ we see that \eqref{max2q}  goes to zero for $q>1$.
Multiplying out the right side of \eqref{condition on hn}, we have that any term with a factor of $1_{\{|\cc_n(1,j_{2})|\le h_n\}}$ must have $q(1)=1$
in order to contribute to the limiting sum. It should perhaps be noted that since we must have that $q(1)=1$, for terms with a factor of $1_{\{|\cc_n(1,j_{2})|\le h_n\}}$,
the permuting/concatenating of sub-cycles which form $L$ and $\tilde L$ is not needed.

We now write
$$\cc_n(j_{k},j_{k+1})=\cc_n(j_k,j_{k+1})(1_{\{|\cc_n(j_k,j_{k+1})|\le h_n\}}+1_{\{|\cc_n(j_k,j_{k+1})|>h_n\}})$$
for all factors in all terms of \eqref{m2p} and \eqref{condition on hn}.
For fixed $j_2\equiv L_1$, we will categorize terms containing the factor $1_{\{|\cc_n(1,j_{2})|\le h_n\}}$ by the number of other factors in the term
which are of the form $$|\cc_n(1,j_{k})|^21_{\{|\cc_n(1,j_{k})|\le h_n\}}\quad\text{for any }\ k.$$  There are at most $p/2$ such factors.
In particular, consider terms of \eqref{condition on hn} which include the factor $|\cc_n(1,\tL_2)|^{2}1_{\{|\cc_n(1,\tL_{2})|\le h_n\}}$.
The above procedure on our cycle rooted at $1$  is repeated on the cycle $\tL=: \tL^{(1)}$, which is also rooted at $1$.  In other words,
 we fix the value of $\tL_2$ and consider cycles such that the edge $(1,\tL_2)$ is crossed exactly once in each direction. We remove
these $2$ edges from $\tL$ leaving us with
two sub-cycles $L^{(2)}$ and $\tL^{(2)}$ rooted at $L_1^{(2)}:=\tL_2^{(1)}$ and $\tL_1^{(2)}:=1$.
We then repeat the procedure on the cycle $\tL^{(2)}$ to get two more sub-cycles $L^{(3)}$ and $\tL^{(3)}$, and continue this process until all edges of the form $(1,\cdot)$ or $(\cdot,1)$ are ``removed''. Thus, for any term containing $1_{\{|\cc_n(1,j_{2})|\le h_n\}}$ there is a corresponding list of cycles $(L^{(1)}, L^{(2)},\ldots, L^{(M)})$.  The list is of length $M\le p/2$ where $M$ depends on the term (thus terms are categorized by their associated $M$ value), and each cycle in the list is rooted at a different vertex in $\{2,\ldots,n\}$.
Let ${\mathbb{L}_M(n)}$ denote the set of all possible lists of cycles of length $M$.

\includegraphics[scale=.33]{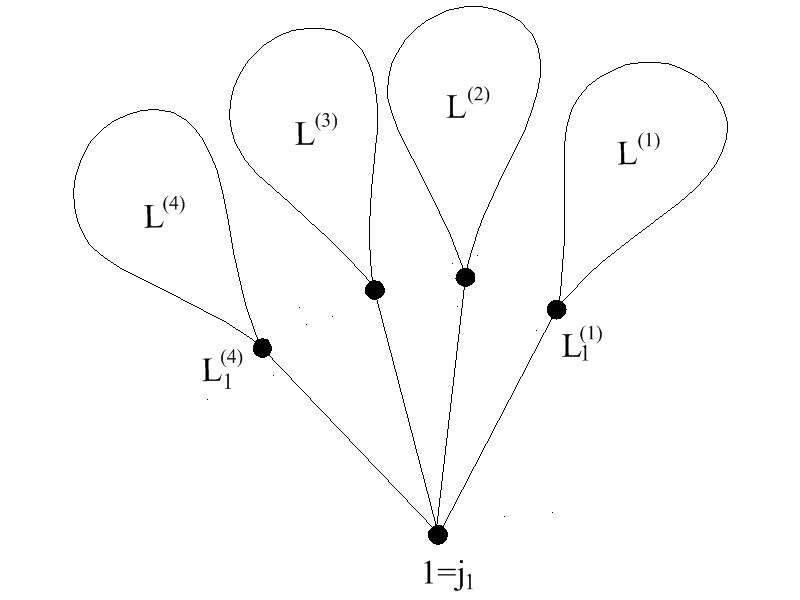}

Finally, recalling that $L_1^{(1)}\equiv j_2$, the sum of all contributing terms in \eqref{condition on hn}
can be written in the form
\bea \nn
\sum_{M=0}^{p/2} \sum_{{\mathbb{L}_M(n)}}\prod_{i=1}^M\(\E\[ |\cc_n(1,L_1^{(i)})|^{2}1_{\{|\cc_n(1,j_{2})|\le h_n\}})\]\E\[\prod_{i=1}^M\cc_n(L_1^{(i)}, L_2^{(i)})\cdots \cc_n(L_{s^{(i)}}^{(i)},L_1^{(i)})\]\).
\eea
Summing over the possible first coordinates of each cycle in the list of cycles, $L^{(i)}_1\in\{2,\ldots, n\}$, and taking the limit gives us
\bea\label{last moment eq}
 \lim_{n\to\infty}\sum_{M=0}^{p/2}\sum_{(L^{(1)},\ldots,L^{(M)})\in{\mathbb{L}_M(n)}}\sigma^{2M}\E\[\prod_{i=1}^M\cc_n(L_1^{(i)}, L_2^{(i)})\cdots \cc_n(L_{s^{(i)}}^{(i)},L_1^{(i)})\].
\eea
Let $(\cc^{\sigma,1}_n)$ be matrices which are modified using only the first step of the replacement procedure, i.e., where only a single cord to infinity (from 1) has been substituted.
Using the fact that
$$|\cc_n(j_k,j_{k+1})|=|\cc^{\sigma,1}_n(j_k,j_{k+1})| \ \text{ on the event }\ \{|\cc_n(j_k,j_{k+1})|> h_n\},$$
a relatively straightforward calculation of $M_{p}(\E\Spec_{\tilde\cc_n^\sigma})$ using \eqref{m2p} also gives \eqref{last moment eq} by
\begin{itemize}
\item[(a)] conditioning on the number of times that a given cycle rooted at $1$ crosses the cord from 1 to infinity (in either direction) to be $2M$, and
\item[(b)] for a fixed set of loops $L^{(1)},\ldots, L^{(M)}$ in \eqref{last moment eq} with different roots, one can identify their different roots with one single root. This single root should be thought of as the vertex at infinity which is connected to 1 by the cord in (a) above. One need only check that the two configurations of loops give the same value for
    the expression
    \be\label{same root}
      \lim_{n\to\infty}\E\[\prod_{i=1}^M\cc_n(L_1^{(i)}, L_2^{(i)})\cdots \cc_n(L_{s^{(i)}}^{(i)},L_1^{(i)})\].
      \ee
\end{itemize}
There is a slight subtlety regarding the invariance of \eqref{same root} under the identification of roots. The subtlety is that the dependence structure of edges crossed in \eqref{same root} is changed under the
identification of roots.
However, note that we can approximate the L\'evy measure by a sum of Dirac point measures, and without loss of generality, we will assume it has this form.  Then, it turns out that the dependence structure of edges crossed in \eqref{same root} does not affect the value of \eqref{same root} since (i) the dependence structure only changes on the event that the various edges crossed have a common weight $\lambda$, and (ii) in this event, the $2q$th moment of $\lambda$ times a Rademacher random variable is $\lambda^{2q}$.  Thus, for example, the product of the variances of two independent $\lambda$-scaled Rademachers is exactly the fourth moment of a single $\lambda$-scaled Rademacher.

The proof of the theorem is now complete for the first step of the replacement procedure.  Equivalence of moments for other steps in the replacement procedure follows similarly, and the rest of the proof is left as an exercise.




\end{proof}

\begin{remark} When $\Pi$ is trivial, all the $q_i$'s in \eqref{qi's} are {\it all} equal to $2$. This leads to the well-known fact that \eqref{m2p} is the number of Dyck words of length $2p$ which is just the $p$th Catalan number $$c_{p}=\frac{(2p)!}{p!(p+1)!}.$$
\end{remark}

We next have a result which relates the moments of the matrix entries to the moments of the L\'evy measure. Both sets of moments are also related to the moments of the LSD
using \eqref{qi's}; moreover, together with the proposition below, \eqref{qi's} proves existence of the limit in \eqref{moments}.
\begin{prop}[Triangular array moments are related to L\'evy measure moments]\label{lem:moments}
Suppose that $\{\cc(n,k), 1\le k\le n\}_{n\in\N}$ is a triangular array of random variables which are i.i.d. in each row, and for which
$\sum_{k=1}^{n}|\cc(n,k)|^2$ converges weakly as $n\to\infty$
to an infinitely divisible law with subordinator characteristics $(\sigma^2,\Pi_s)$.
If the random variables are uniformly bounded,
\be\label{a:ub}
|\cc(n,k)|\le\tau \quad\text{for all } n \text{ and } k,
\ee
 then
$$\lim_{n\to\infty}n \E |\cc(n,1)|^2=\sigma^2+M_1(\Pi_s)\, ,$$
and for $p>1$
$$\lim_{n\to\infty}n \E |\cc(n,1)|^{2p}=M_{p}(\Pi_s).$$
\end{prop}
\begin{proof}
Set $X_n:=|\cc(n,1)|^2$ with characteristic function $\varphi_{X_n}$.
The characteristic function of $$\lim_{n\to\infty}\sum_{k=1}^n |\cc(n,1)|^2\ed X$$ in \eqref{psi2} takes the form
\be
\varphi_X(\theta)=\exp\( i\theta\sigma^2+\int_0^{\tau^2} (e^{i\theta x}-1)\, \Pi_s(dx)\)
\ee
and by convergence in distribution of the row sums {and}
 Lemma 5.8 in \cite{kallenberg2002foundations}, 
\be\nn
\lim_{n\to\infty}n(\varphi_{X_n}-1)=i\theta\sigma^2+\int_0^{\tau^2} (e^{i\theta x}-1)\, \Pi_s(dx)
\ee
uniformly in $\theta$ on compact subsets of $\R$.
Since the {$\{X_n\}$} are bounded and since $\Pi_s$ has bounded support we may expand both sides in terms of power series and switch summations with integrals.  This gives us
\be
\lim_{n\to\infty}n \sum_{k\ge 1} \frac{(i\theta)^k\E X_n^k}{k!}= i\theta\sigma^2+\sum_{k\ge 1} \int_0^{\tau^2} \frac{(i\theta x)^k}{k!}\, \Pi(dx)
\ee
uniformly on compact subsets, from which the lemma follows.
\end{proof}

To verify the ``moment problem'' required to use Proposition \ref{lemma:same}, we adapt arguments from \cite{bauer2001random, khorunzhy2004eigenvalue, zakharevich2006generalization}.
 Let $Q_{p}$ be the set of $(q_1,\ldots,q_\ell)$ such that $q_i\in\N$, $\ \sum_{i=1}^\ell q_i = p$,
and
$$q_1\ge q_2\ge\cdots\ge q_\ell.$$ Also, fix a sequence of distinct colors $\{K_i\}_{i=0}^\infty$. We define
$T((q_1,\ldots,q_p))$  to be the number of colored rooted trees which satisfy
\begin{itemize}
\item There are $p+1$ vertices.
\item There are exactly $q_i$ vertices of color $K_i$ with the root being the only vertex of color $K_0$.
\item If $u$ and $v$ are the same color then the distance from $u$ to the root is equal to the distance from $v$ to the root.
\item If  $u$ and $v$ have the same color then so do their parents.
\end{itemize}
Define $$\II_{p,\ell}:=\sum_{(q_1,\ldots,q_\ell)\in Q_p}T((q_1,\ldots,q_\ell)).$$

\begin{prop}[LSD determined by its moments] \label{prop:moment bound}
Under assumption \eqref{finite nu moments},
\be
M_{2p}(\E\Spec_{\cc_\infty})\le\tau^{2p}\sum_{\ell}\II_{p,\ell}\(M_2(\Pi)+\Pi([-1,1]^c)+\sigma^2\)^\ell,
\ee
and thus $\E\Spec_{\cc_\infty}$ exists and is determined by its moments.
\end{prop}
\begin{proof}
By splitting the support of $\Pi$ into $[-1,1]$ and its complement, note that $M_{2q}(\Pi)\le M_2(\Pi)+\tau^{2q}\Pi([-1,1]^c)$.
 Also, without loss of generality, $\tau\ge 1$. We use {Proposition} \ref{lem:moments} in conjunction with the argument of \cite[Thm. 2]{zakharevich2006generalization} (see also \cite[Sec. 5.3]{bauer2001random} and \cite[Sec. IV]{khorunzhy2004eigenvalue})
to get
\bea\label{zb}
&&M_{2p}(\E\Spec_{\cc_\infty})=\nn\\
&&\nn\lim_{n\to\infty} \sum_{(q_1,\ldots,q_\ell)\in Q_p}T((q_1,\ldots,q_\ell))n\E(|\cc_n(1,2)|^{2q_1})\cdots n\E(|\cc_n(1,2)|^{2q_\ell})\\
&&\le \nn\sum_{(q_1,\ldots,q_\ell)\in Q_p}T((q_1,\ldots,q_\ell)) (M_{2q_1}(\Pi)+\sigma^{2}) \cdots (M_{2q_\ell}(\Pi)+\sigma^2)\\
&&\le \nn \tau^{2p}\sum_{(q_1,\ldots,q_\ell)\in Q_p}T((q_1,\ldots,q_\ell)) \(M_2(\Pi)+\Pi([-1,1]^c)+\sigma^{2}\)^\ell\\
&&=\tau^{2p}\sum_{\ell}\II_{p,\ell}\(M_2(\Pi)+\Pi([-1,1]^c)+\sigma^2\)^\ell.
\eea

Next, we use Eq. (9) in \cite{bauer2001random} 
which gives the bound
\be\label{bgb}
\II_{p,\ell}\le c_{p}\SS_{p,\ell}
\ee
(see also Prop. 10 in \cite{zakharevich2006generalization}) where $c_{p}$ is the $p$th Catalan number and
$$\SS_{p,\ell}=\frac 1 {\ell!}\sum_{k=0}^\ell(-1)^{\ell-k}	\dbinom{\ell}{k}k^{2p}$$
is a Stirling number of the second kind. By \eqref{zb}, \eqref{bgb}, and Theorem 30.1 in \cite{billingsley1986probability},
$\E\Spec_{\cc_\infty}$ is determined by its moments if for any $R>0$,
\be\label{eq:moment bound}
\frac{c_{p}}{(2p)!}\sum_{\ell=1}^p R^\ell\SS_{p,\ell}
\ee
 is $o(r^{p})$ for some $r$ as $p\to\infty$, and this is easily verified. For example Section 5.5 of \cite{bauer2001random} shows \eqref{eq:moment bound} is
less than $(p^p+e^{R(p-1)})/(p!(p+1)!)$.
\end{proof}

{
\begin{remark}
In \cite{bauer2001random}, the lower bound $\SS_{2p,\ell}\le \II_{2p,\ell}$ was also established and used to show that the LSD has unbounded support (see also \cite[Prop. 12]{zakharevich2006generalization}).
In our situation, this tells us that the \LK\ ensembles for which the LSD has bounded support are precisely those with only a Wigner portion, i.e., those with
characteristics of the form $(\sigma^2,0,0)$.
\end{remark}}

\section{From local weak convergence to spectral convergence}\label{sec:nowigner}

In this section, to simplify things we restrict our attention to
random conductance matrices $\cc_n$ with {\it real} entries.
The goal of this
section is to present Theorem \ref{thm:lwc to spectrum} which uses
strong resolvent convergence to connect the notions of local weak
convergence and weak convergence of ESDs. Theorem \ref{thm:lwc
to spectrum} below is similar to \cite[Theorem 2.2]{bordenave2011spectrum}  (see also
\cite{bordenave2010resolvent, bordenave2011nonHermitian, bordenave2012around}), and its proof
is an adaptation of the arguments there
which treat the symmetric $\alpha$-stable case:
$$(\sigma^2,0,\Pi)=(0,0,\text{sign}(x)\alpha|x|^{-1-\alpha}dx).$$ Here
we replace the $\alpha$-stable L\'evy measure with an arbitrary
symmetric L\'evy measure $\Pi(dx)$ on $\Roo$. In particular, if one assumes self-adjointness of the
limiting operator (which follows from Lemma \ref{SelfAdjointness} below),
then the arguments in this section are enough
to handle Theorem \ref{thm:main2} in the case when $\sigma=0$ and the entries are real.

Let us now present the precise notion of local weak convergence
following the treatment in \cite{aldous2004objective}. Let
$G[\emp]=(V,E)$ be a $\emp$-rooted graph with vertex set $V$ and
edge set $E$ both of which are at most countably infinite. Any
edge-weight
function $\r:E\to\Roo$
 defines
a distance  between any two vertices $u,v\in V$ as \be
\label{distance} d(u,v):=\inf_{\gamma\text{ connects }u,v}
\sum_{e\in \gamma} |\r(e)| \ee where the infimum is over all paths $\gamma$
which connect vertices $u$ and $v$. The distance $d$ naturally turns
$G[\emp]$  into a metric space. We include $\pm\infty$ as a possible
edge-weight where $\pm\infty$ is thought of as the same weight
using the one-point compactification of $\Roo$.

If $G[\emp]$ is connected and undirected and the edge-weight
function $\r$ is such that for every vertex $v$ and every
${r}<\infty$, the number of vertices within distance ${r}$ of $v$ is
finite, then $G[\emp]=(V,E,\r)$ is a {\bf rooted geometric graph}.
Henceforth all graphs will be rooted geometric graphs, and when they
are rooted at the default root $\emp$, we may simply write $G$
instead of $G[\emp]$. The set of all rooted geometric graphs is
written $\GG_\star$.

In the case that the range of $\r$ is  positive and the underlying graph is a tree, we can interpret
$\r$ as assigning resistances to edges.  However, for technical reasons
required by the proofs of our main results, {\it we allow $\r$ to take negative values}.
The possibility of negative weights makes our treatment here differ slightly from
\cite{aldous2004objective}.  But, using the
modulus  in \eqref{distance} nevertheless permits us to reap the
benefits of the metric of \cite{aldous2004objective} on $\GG_\star$.


Let $\NN_{{r},\emp}(G)$ be the ${r}$-neighborhood of $\emp$. This is
the $\emp$-rooted subgraph of $G$ formed by restricting the graph to
the set of all vertices $v\in V$ such that $d(\emp,v)\le {r}$ and
restricting to the set of edges that can be crossed by journeying at
most distance $r$ from the root $\emp$. We say $r$ is a {\it
continuity point} of $G$ if there is no vertex of exact distance $r$
from the root.

\begin{defi}[The topology of $\GG_\star$] We say $\left(G_{n}=(V_n,E_n,\r_n)\right)_{n\in\N}$ converges to $G=(V,E,\r)$ in $\GG_\star$
if for each continuity point $r$ of $G$, there is an $n_r$ such that
$n>n_r$ implies there exists a graph isomorphism $$\pi_n:\NN_{{r},\emp}(G)\to \NN_{{r},\emp}(G_n)$$ which preserves the
root and for which
\be\label{edge limit}
  \lim_{n\to\infty} \r_{n}(\pi_{n}^{-1}(u),\pi_{n}^{-1}(v)) = \r(u,v).
  \ee
\end{defi}

As noted in \cite{aldous2004objective}, the above convergence
determines a topology which turns $\GG_\star$ into a complete
separable metric space. Using the usual theory of convergence in
distribution, one can therefore say that a sequence of random rooted geometric graphs
$\(G_n\)_{n\in\N}\subset\GG_\star$, with distributions $\mu_n$, converge weakly to
$G\in\GG_\star$ with distribution $\mu$ if for all bounded
continuous $f:\GG_\star\to\R$ \be \int_{\GG_\star} f d\mu_n \to
\int_{\GG_\star} f d\mu. \ee Such weak convergence is called {\it
local weak convergence}.

The following connection between local weak convergence and strong resolvent convergence was first noticed in \cite{bordenave2010resolvent} and \cite{bordenave2011spectrum}
in the context of sparse matrices and heavy-tailed matrices, respectively {(see \cite{hora2007quantum} for related arguments).}

\begin{thm}[Local weak convergence implies strong resolvent convergence]\label{thm:lwc to spectrum}
Let $\(\cc_{G_n}\)_{n\in\N}$, which are associated to $\(G_{n}=(V_n,E_n,\r_n)\)_{n\in\N}$ as in \eqref{aPWIST}, be essentially self-adjoint.
Suppose that the graphs converge in the local weak sense to a tree $G=(V,E,\r)$
with respect to the isomorphisms {$(\pi_n)_{n\in\N}$}, and that $\cc_G$ is also essentially self-adjoint.

If for each $u\in V$,
\bea\label{eq:ui}
\lim_{\eps\searrow 0}\lim_{n\to\infty}
\sum_{v\in V_n:v\sim \pi_n(u)}|\cc_{G_n}(\pi_n(u),v)|^21_{\{|\cc_{G_n}(\pi_n(u),v)|^2\le\eps\}}=0 \ \text{a.s.}, \eea
 then for all
$z\in\C_+$, as $n\to\infty$: \be\label{part1}  \la e_\emp,(\cc_{G_n}-zI)^{-1}e_\emp\ra
\weak  \la e_\emp,(\cc_{G}-zI)^{-1}e_\emp\ra .\ee
\end{thm}

\begin{remark} By Proposition \ref{Kall cor15.16},
condition \eqref{eq:ui}
simply says that $\sigma^2=0$ in \eqref{LKcondition2}.
\end{remark}

Once one checks the local weak convergence of $\(G_n[1]\)_{n\in\N}$
to a PWIT$(\Pii)$ and verifies self-adjointness, then the above result essentially handles the case
where the Wigner component vanishes. Let us briefly outline this.
First of all $\sigma=0$ will imply condition \eqref{eq:ui}. Next,
recall that the \CS\ transform (or simply Stieltjes transform) is defined as \be
\label{def:Stieltjes} S_\mu(z):= \langle \mu, (x-z)^{-1} \rangle =
\int_{\R} \frac{\mu(dx)}{x-z}, \ \ \ z\in\C\backslash\R.\ee
Recall from
\eqref{esd} that
 $\Spec_{\cc_n}$ is the ESD of $\cc_{n}$.
Using the fact that entries in $\cc_n$ are i.i.d., \be
\label{eq:resolvent identity} S_{\E\Spec_{\cc_n}}(z) = \E
S_{\Spec_{\cc_n}}(z)= \frac 1 {n} \E \tr (\cc_{n}-zI)^{-1}=\E
(\cc_{n}-zI)^{-1}(1,1). \ee Therefore, by \eqref{eq:resolvent
identity}, the above theorem, and a bound on the modulus of the Green's function
$$|(\cc_{n}-zI)^{-1}(1,1)|\le (\Im z)^{-1}$$
for $z\in\C\backslash\R$, we obtain convergence of {$(S_{\E\mu_{\cc_n}})_{n\in\N}$} to $S_{\E \mu_{\cPWIST}}$ where $G_\infty$ is a
PWIT($\Pii$). Lemma \ref{lem:2.4.4}, which tells us that the
\CS\ transform determines the LSD, then implies weak convergence of
the expected ESDs (since $e_\emp$ has unit norm, the limit is a probability measure). A concentration of
measure argument from \cite{guntuboyina2009concentration}, Lemma
\ref{lem:com} below, extends this to a.s. weak convergence for the
random ESDs.

For the proof of Theorem \ref{thm:lwc to spectrum} we need a lemma
which appears as Thm VIII.25 in \cite{reed1980methods}. We state it without proof.
\begin{lem}[Strong resolvent convergence characterization]\label{reedlemma}
Suppose $\cc_{n}$ and $\cc_\infty$ are self-adjoint operators on
$L^2(V)$ with a common core $\DD$ (for all $n$ and $\infty$). If
$$\cc_{n}\varphi\to \cc_\infty\varphi \quad\text{in } \ L^2(V),$$ for
each $\varphi\in \DD$, then  $\cc_n$ converges to $\cc_\infty$ in the
strong resolvent sense.
\end{lem}

\begin{proof}[Proof of Theorem \ref{thm:lwc to spectrum}]
To match the setting for which we employ this theorem, let the vertex set of $G_n$ be a subset of $\N$ and the vertex set of $G$ be $\N^F$.
By assumption, the local weak limit of $\(G_n\)_{n\in\N}$ is the tree
$G$, with respect to the mappings \be\label{pi2}
\pi_n:\N^F\to V_n\subset\N\ee which are
injective when restricted to some random subset of $\N^F$ with
the same cardinality as $V_n$. By the Skorokhod representation theorem we will in
fact assume that this weak convergence in $\GG_\star$ is almost sure
convergence on some probability space. Note that when the sequence $\(\cc_{G_n}\)_{n\in\N}$ is a sequence of $n\times n$ \LK\ matrices, one may
set $V_n=\{1,\ldots,n\}$, however in general $V_n$ may even be infinite (in which case it is just $\N$).

Since $\N^F$ is countable we can fix some bijection with $\N$ and
think of $V_n$ as a subset of $\N^F$. In this case, the
maps $\pi_n$ can each be extended to (random) bijections from $\N^F$
to $\N$, and abusing notation we write $\pi_n$ for these extensions.
The essentially self-adjoint operators $\cc_{G_n}$ extend to self-adjoint operators on
$L^2(\N^F)$, using the core $\core$ consisting of vectors with
finite support, by defining \be \label{eq:sa} \la e_{u},
\cc_{G_n}e_{v} \ra:=
\begin{cases}
& \cc_{G_n}(\pi_n(u),\pi_n(v))\quad \text{if } \{\pi(u), \pi(v)\}\subset V_n\\
& 0 \quad\quad \quad \quad\quad\quad\quad\quad\  \text{otherwise}.
\end{cases}
\ee
By assumption, the closure
of $\cc_{G}$ is also self-adjoint using the core $\core$.
Again abusing notation, we identify this closure with $\cc_{G}$.

By local weak convergence and Skorokhod representation, we have that
almost surely \be\label{as conv} \la {e}_u, \cc_{G_n}{e}_{v}\ra \to
\la{e}_u, \cc_{G}\, {e}_v\ra. \ee By Lemma \ref{reedlemma},
we are left to show that \be\nn
\sum_{u\in\N^F}|\la{e}_u,\cc_{G_n}{e}_{v}\ra-\la{e}_u,
\cc_{G}\, {e}_v\ra|^2\to 0 \ee almost surely, as
$n\to\infty$. This follows from the Vitali convergence theorem since
\eqref{as conv} provides almost sure convergence and
 \eqref{eq:ui} provides uniform square integrability.
\end{proof}

A common tool for showing local weak convergence is the following
lemma about Poisson random measures which is similar to \cite[Lemma
4.1]{steele2002minimal}.


\begin{lem}[Convergence to a Poisson random measure]\label{lem:ppp}
Suppose $\{\cc(n,k), 1\le k\le n\}_{n\in\N}$ is a triangular array of real random variables which are i.i.d. in each row, and for which
$\sum_{k=1}^{n}\cc(n,k)$ converges in law, as $n\to\infty$,
to an $ID(\sigma^2,b,\Pi)$ random variable.  Then as $n\to\infty$
$$
\sum_{k=1}^n \delta_{\cc(n,k)}
$$
converge vaguely, as measures on $\Roo$, to a Poisson random measure
$\eta$ with intensity $\E\eta=\Pi$.
\end{lem}
\begin{proof}[Proof of Lemma \ref{lem:ppp}]
Note that any L\'evy measure $\Pi$ is also a Radon measure on $\Roo$.
Even though there is a possible singularity at $0$, this is no
concern since $0\notin \Roo$. Therefore, by the basic convergence
theorem of empirical measures to Poisson random measures  (see
Theorem 5.3 in \cite{resnick2007heavy}) we need only check that
\be\nn n\P(\cc_n(1,2)\in \cdot)\vague \Pi \ee vaguely as measures on
$\Roo$. This follows from Proposition \ref{Kall cor15.16}.
\end{proof}
\begin{remark}
It is instructive to recognize that  the L\'evy characteristics
$\sigma^2$ and $b$ bear no influence on the above lemma, and
consequently bear no influence on local weak convergence of the
associated graphs. This is because vague convergence pushes any
affect they have to the point $0$ which is not in $\Roo$. This
essentially tells us that $b$ has no effect on the LSD which is one reason why we were allowed to set it to 0 (this
statement is made rigorous by Theorem \ref{thm:main2}). The same is
not true for $\sigma^2$ since we must have $\sigma=0$ in order to
satisfy \eqref{eq:ui} (uniform square integrability) and therefore
to use Theorem \ref{thm:lwc to spectrum}. However, after one applies
the replacement procedure, \eqref{eq:ui} will once again be
satisfied.
\end{remark}
The following proposition utilizes Lemma
\ref{lem:ppp} to show local weak convergence to a PWIST.  It is a variant of
results in \cite[Sec. 3]{aldous1992asymptotics} (see also
\cite{aldous2001zeta, steele2002minimal, bordenave2011spectrum}).

\begin{prop}[Local weak convergence to a PWIST] \label{prop:LWC} %
Let $G_n[1]$ have
conductances $\{\cc^{\sigma}_n(j,k)\}_{j,k}$ which are modified \LK\
matrices with characteristics $(\sigma^2,0,\Pi)$ (modified as in Section \ref{replacement}). Then the local weak limit of
$(G_n[1])_{n\in\N}$ is a PWIST$(\sigma, \Pii)$.
\end{prop}

\begin{proof} We follow \cite[Sec. 3]{aldous1992asymptotics} and \cite[Sec. 2.5]{bordenave2011spectrum}.
For each fixed realization of the
$\{\cc^{\sigma}_n(j,k), 1 \leq j,k \leq n\}$ we
consider their reciprocals, i.e., the resistances
$$\{\r^{\sigma}_n(j,k), 1 \leq j,k \leq n\}.$$
For
any $B,H\in\mathbb{N},$ such that
$$\sum_{\ell=0}^HB^\ell\leq n,$$
we define a rooted geometric subgraph $G_n[1]^{B,H}$ of
$G_n[1]$, whose vertex set is in bijection with a $B$-ary tree of
depth $H$ rooted at $1$. Let $V_n:=\{1,\ldots,n\}$. The bijection provides a partial index of
vertices of $G_n[1]$ as elements in \be J_{B,H} = \bigcup_{\ell=0}^H
\{1,\ldots, B\}^\ell\subset{\mathbb{N}_0^F}\ee where the indexing is
given by an injective map $$\pi_n:J_{B,H}\to V_n.$$ The map
$\pi_n$ easily extends to a bijection from some subset of
{$\mathbb{N}_0^F$} to $V_n$ and thus can be thought of as restrictions of the
maps of \eqref{pi2}.

We set $I_\emp= \{ 1 \}$ and set the preimage/index of the root $1$
to be $ \pi_n^{-1} (1) = \emp$. We next index the $B$ vertices in
$V_n \setminus I_{\emp}$ which have the $B$ smallest absolute values
among
$\{\r^{\sigma}_n(1,k)\}_{2 \leq k \leq n}$. The
$k$th smallest absolute value is given the index $\emp k =
\pi_n^{-1} (v)$, $1 \leq k\leq B$. As in the discussion preceding
\eqref{nof},  we have written the vector $\emp k$ using
concatenation. Breaking ties using the lexicographic
order, this defines the first generation.

Now let $I_1$ be the union of $I_\emp$ and the $B$ vertices that
have been selected. If $H\geq2$, we repeat the indexing procedure
for the vertex indexed by $\emp 1$ (the first child of $\emp$)  on
the set $V_n \setminus I_1$. We obtain a new set $\{11,\ldots,1B\}$
of vertices sorted by their absolute resistances. We define $I_{2}$
as the union of $I_1$ and this new collection. Repeat the procedure
for $\emp 2$ on $V_n \setminus I_{2}$ and obtain a new set
$\{21,\ldots,2 B\}$.  Continuing on through $\{B1,\ldots ,BB\}$, we
have constructed the second generation, at  depth $2$, and we have
indexed a total of $(B^{3} - 1)/(B-1)$ vertices. The indexing
procedure is repeated through depth $H$ so that $(B^{H+1} -
1)/(B-1)$ vertices are sorted. Call this set of vertices $V_n^{B,H}
= \pi_n (J_{B,H})$. The  subgraph of $G_n[1]$ generated by the
vertices $V_n^{B,H}$ is denoted $G_n[1]^{B,H}$ (by ``generated'' we
mean that we include only edges with endpoints in the  specified
vertex set). It is the modification of $G_n[1]$ such that any edge
with at least one endpoint in the complement of $V_n^{B,H}$ is given
an infinite resistance. In $G_n[1]^{B,H}$, the elements of
$\{u1,\ldots,{u}B \}$ are the children of $u$.

Note that while the vertex set $V_n^{B,H}$ has a natural tree
structure, $G_n[1]^{B,H}$ is actually a subgraph of a complete
graph which may not be a tree.

Let $G_\infty[\emp]$ be a  PWIST$(\sigma, \Pii)$, or a PWIT($\Pii$)
if $\sigma=0$, and write $G_\infty[\emp]^{B,H}$ for the finite
rooted geometric graph obtained by the sorting procedure just
described. Namely, $G_\infty[\emp]^{B,H}$ consists of the subtree
with vertices of the form $u\in J_{B,H}$, with resistances between
these vertices inherited from the infinite tree. If an edge is not
present in $G_\infty[\emp]^{B,H}$, we may think of it as being
present but having infinite resistance.

Since the conductances $\{\cc^{\sigma}_n(j,k)\}$ by definition are real with a symmetric distribution, we may without loss of generality replace $\sum_{j=1}^n \pm|\cc_n(1,j)|$ with $\sum_{j=1}^n \cc_n(1,j)$  in \eqref{LKcondition}.
We use Lemma \ref{lem:ppp} on the unmodified matrices (with real and symmetrically distributed entries) to conclude that $\sum_{k=1}^n \delta_{\cc_n(1,k)}$
converges vaguely to a Poisson random measure with intensity $\Pi$.
{For $h_n\to 0$}, the truncation
$\cc(n,k)1_{|\cc(n,k)|\le h_n}$  does not affect this vague
convergence. Note that besides the random resistances on edges given by the Poisson random measure, there is also one
more nonrandom resistance given by  the replacement procedure (for $n$ large enough), and the value is always $1/\sigma$.
It is easily verified that the property in \eqref{edge limit} is satisfied
by each edge $(u,v)$ of the tree $G_\infty[\emp]$.

It remains to check that for each $B$ and $H$, our maps $\pi_n$ are
graph isomorphisms for $n$ large enough. In other words, we must
check that for each edge in $G_\infty[\emp]^{B,H}$ with an infinite
resistance, the corresponding edges of $\(G_n[1]^{B,H}\)_{n\in \N}$ (for
$n$ large enough), must have resistances which diverge to infinity.
The divergence of these resistances to infinity follows from a
standard coupling argument which shows that these resistances
stochastically dominate i.i.d. variables with distribution
$\r_n(1,2)$ which clearly diverges as $n\to\infty$ (see for example,
Lemma 2.7 in \cite{bordenave2011spectrum}).
\end{proof}

\section{Proofs of the main results}\label{sec:proofs}

In the case that a \LK\ ensemble $\(\cc_n\)_{n\in\N}$ {is real and} has characteristics of the form $(0,0,\Pi)$, then results of Section \ref{sec:nowigner} (Theorem \ref{thm:lwc to spectrum}, Proposition \ref{prop:LWC}) imply
the existence of the LSD in expectation.  On the other hand,
if $|\cc_n(1,2)|$ is a.s. uniformly bounded in $n$, Proposition \ref{prop:moment bound}
proves the existence of the LSD in expectation.  

%
We turn now to the general assumptions of Theorems \ref{thm:main} and \ref{thm:main2}.
 Before proving the main results, we have three preliminary lemmas.
Our first preliminary lemma allows us to extend from convergence in expectation to almost sure
convergence. It is a concentration of measure result first
noticed in \cite[Theorem 1]{guntuboyina2009concentration}  and later in
\cite[Lemma C.2]{bordenave2011nonHermitian}.  
 We state it here
without proof.

\begin{lem}[Concentration for ESDs]\label{lem:com}
Let $\hh_n$ be an $n\times n$ Hermitian matrix whose rows are
independent (as vectors).  For every real-valued continuous $f(x)$ going to 0 as $x\to\pm\infty$ such that $\|f\|_{\text{TV}}\le 1$, and for every
$t\ge 0$,
$$\P\(\lb \int_\R f \, d\mu_{\hh_n} - \E\int_{\R} f \,d\mu_{\hh_n}\rb\ge t\)\le 2\exp\(-nt^2/2\)$$
\end{lem}

The next lemma verifies the self-adjointness of PWISTs required to use Theorem \ref{thm:lwc to spectrum}.

\begin{lem}[Self-adjointness of PWIST operators]\label{SelfAdjointness}
Suppose $G_\infty[\emp]=(V_\infty,E_\infty,\r_\infty)$ is a PWIST($\sigma,\Pii$). Then the associated random conductance operator $\cc_{G_\infty}$ on $L^2(V_\infty)$, as defined in
\eqref{aPWIST}, is essentially self-adjoint.
\end{lem}

\begin{proof}
Denote the children of the root $\emp$ of a PWIST($\sigma,\lambda_\Pi$) by $\N=\N[\emp]$ where they are ordered according to the absolute value of the conductances  on the edges where the edge to $1$ has the largest absolute conductance. For $\kappa>0$ as chosen below, define the random variable
$$\tau_\emp:= \inf
\{J: \sum_{j=J}^\ff |\cc_{G_\infty}(\emp,j)|^2\le\kappa\}
$$
and define the i.i.d. random variables $\{\tau_v\}$ similarly by considering the conductances on $\N[v]$ (in place of $\N[\emp]$).
By the integrability conditions on  L\'evy measure $\Pi$, we
may choose $\kappa$ large enough so that
 $\E\tau_\emp<1$.
We may therefore employ the proof of Proposition A.2 in \cite{bordenave2011spectrum} to show that for any PWIST, $G_\infty=(V_\infty, E_\infty,\r_\infty)$,
 there is a constant $\kappa>0$ and a
sequence of connected finite increasing subsets $\(V_n\)_{n\in\N}$ whose union is $V_\infty$, and
such that for all $n$ and $u\in V_n$
$$\sum_{v\notin V_n:v\sim u}|\cc_{G_\infty}(u,v)|^2<\kappa.$$
Finally, the existence of such a $\kappa$ allows us to use Lemma A.3 in \cite{bordenave2011spectrum} to conclude that any PWIST is essentially self-adjoint. Thus its closure is self-adjoint.
\end{proof}

The final preliminary lemma, similar to arguments in \cite{benarous2008spectrum}, is used to show that the truncation in \eqref{finite nu moments} does not
effect the LSD too much.
For any truncation level $\tau>0$, let $\tau\cc_n$ be a matrix with entries given by
\be\label{truncated matrices}
\tau\cc_n(j,k):=\cc_n(j,k)1_{\{|\cc_n(j,k)|\le \tau\}}.
\ee
\begin{lem}[Large deviation estimate for the rank of a truncation]\label{LDP lemma}
For every $\eps>0$ and $\tau\gg 0$ (large enough depending on $\eps$), there is a  $\delta_{\eps,\tau}>0$  such that
$$\P(\text{rank}(\cc_n-\tau\cc_n)/n\ge \eps)\le \exp\(-\delta_{\eps,\tau}n\).$$
\end{lem}
\begin{proof}
Fix $\eps>0$ and consider $\tau$ large enough (specified below).
Define the events $$U_{jn}:=\{\text{there exists } k \text{ such that }k>j\text{ and }|\cc_n(j,k)|>\tau\}$$
$$L_{jn}:=\{\text{there exists } k \text{ such that }k<j\text{ and }|\cc_n(j,k)|>\tau\}$$
and note that
\be
\text{rank}(\cc_n-\tau\cc_n)\le \sum_{j=1}^n \(1_{U_{jn}}+1_{L_{jn}}\).
\ee
We split rows of the matrix along the diagonal to handle the dependence (due to the self-adjointness requirement) among the indicator random variables:
\bea\label{split dependence}
\nn\P(\text{rank}(\cc_n-\tau\cc_n)\ge 2n\eps)&\le& \P\(\sum_{j=1}^n 1_{U_{jn}}\ge n\eps\)+\P\(\sum_{j=1}^n1_{L_{jn}}\ge n\eps\)\\
&\le& 2\P\(\sum_{j=1}^n 1_{U_{jn}}\ge n\eps\)\nn\\
&\le& 2\P\(\sum_{j=1}^n 1^{(j)}_{U_{1n}}\ge n\eps\)
\eea
where $\{1^{(j)}_{U_{1n}}\}_{j=1}^n$ are independent copies of $1_{U_{1n}}$.
The last step follows since  the independent variables $\{1_{U_{jn}}\}_{j=1}^n$ are each stochastically dominated by $1_{U_{1n}}$.

Since the triangular array $\{\cc_n(1,k), 1\le k\le n\}_{n\in\N}$ satisfies \eqref{LKcondition},
$$\lim_{n\to\infty}\P(U_{1n})=1-\exp\{-\Pi([\tau,\infty))\},$$
so we may choose $\tau$ large enough
so that $$\sup_{n}\P(U_{1n})=p<\epsilon.$$  The lemma follows by applying a standard large deviation estimate for i.i.d. Bernoulli($p$) random variables
to  the right side of \eqref{split dependence}.
\end{proof}

This last lemma is used in conjunction with a metric which is compatible with weak convergence. Let
 $$\|f\|_{\LL}:=\sup_{x\neq y}\frac{|f(x)-f(y)|}{|x-y|}+\sup_{x}|f(x)|$$
 Lemma 2.1 in \cite{benarous2008spectrum} says the following variant of the Dudley distance gives a topology which is compatible with weak convergence:
\be\label{dudley}
d_1(\mu,\nu):=\sup_{\|f\|_{\LL}\le 1, f\uparrow}\lb\int f \,d\mu-\int f\,d\nu\rb.
\ee
Moreover, Lidskii's estimate (see Eq. 8 in \cite{benarous2008spectrum}) implies
\be\label{lidskii}
d_1(\mu_{\cc_n},\mu_{\tau\cc_n})\le \frac{\text{rank}(\cc_n-\tau\cc_n)}{n}.
\ee

\begin{proof}[Proof of Theorems \ref{thm:main} and \ref{thm:main2}]
Let us first state some simplifications for the task of showing that the LSD exists
as a weak limit, almost surely.

First of all, by the Borel-Cantelli lemma and Lemma \ref{lem:com}, it is enough to
show weak convergence of $\(\E\Spec_{\cc_n}\)_{n\in\N}$ to $\E\Spec_{\cc_\infty}$.
Next, by exchangeability, it is enough to show weak convergence of the expected spectral measures associated to the basis vector $e_1$.
Finally, by Lemma \ref{lem:2.4.4}, it is equivalent to show convergence
of the \CS\ transforms of these expected spectral measures for each $z\in\C_+$ (the limit will be a probability measure since
it is the spectral measure associated to a unit vector).

 Choose a \LK\ ensemble $\(\cc_n\)_{n\in\N}$ and
let $\(\tau_m\)_{m\in\N}$ be a sequence of positive truncation levels which go to infinity.
For each truncation level $\tau_m$, consider a new sequence of matrices $\(\tau_m\cc_n\)_{n\in\N}$ given by \eqref{truncated matrices}. Recalling our choice of $h_n$ from Section \ref{replacement},
we also consider their modifications $\(\tau_m\cc_n^\sigma\)_{n\in\N}$ (truncation occurs before modification).

Fix $m$. Each modified matrix sequence $\(\tau_m\cc_n^\sigma\)_{n\in\N}$ satisfies the hypotheses of Proposition \ref{prop:LWC},
thus the associated graphs have a PWIST($\sigma,\lambda^{(m)}_\Pi$) as their local weak limit as $n\to\infty$, where
$\lambda^{(m)}_\Pi$ is the intensity $\lambda_\Pi$ restricted to the set $$(-\infty,-1/\tau_m]\cup[1/\tau_m,\infty).$$
The closure of the associated limiting operator
is self-adjoint by Lemma \ref{SelfAdjointness}.
Moreover, by  Proposition \ref{Kall cor15.16} and the properties of the
replacement procedure, we have for each $j\in\N$ that
\be
\lim_{\eps\searrow 0}\lim_{n\to\infty}
\sum_{k=1}^n
\Var\( \ss^{\sigma}_n(j,k)1_{\{|\ss^{\sigma}_n(j,k)|\le
\eps\}}\)=0 \ee which is equivalent to \eqref{eq:ui} since the
entries $\ss^{\sigma}_n(j,k)1_{\{|\ss^{\sigma}_n(j,k)|\le
\eps\}}$ have a real  distribution which is symmetric for $\epsilon<\sigma$ (the truncation $\tau_m$ is unnecessary due to $1_{\{|\ss^{\sigma}_n(j,k)|\le
\eps\}}$).

By the above considerations, we may use Theorem \ref{thm:lwc to spectrum} and the
argument below \eqref{eq:resolvent identity} to conclude Theorem \ref{thm:main2} for each sequence $\(\tau_m\cc_n^\sigma\)_{n\in\N}$.
Thus, the expected LSD of $\(\tau_m\cc_n^\sigma\)_{n\in\N}$, denoted by $\E\mu_{\tau_m\cc_\infty^\sigma}$, is
the expected spectral measure at $e_\emp$ for the self-adjoint random conductance operator $\tau_m\cc_\infty^\sigma$
associated to a PWIST($\sigma,\lambda^{(m)}_\Pi$).

Now take the local weak limit of the PWIST($\sigma,\lambda^{(m)}_\Pi$) graphs as $m\to\infty$.  Since these graphs are truncations
of a PWIST($\sigma,\lambda_\Pi$), it is clear that their local weak limit is just a PWIST($\sigma,\lambda_\Pi$). We may therefore
apply Theorem \ref{thm:lwc to spectrum} once more to conclude that the expected spectral measures at $e_\emp$ of
the PWIST($\sigma,\lambda^{(m)}_\Pi$) operators converge weakly to the expected spectral measure at $e_\emp$ of a PWIST($\sigma,\lambda_\Pi$)
operator which we denote by $\E\mu_{\cc^\sigma_\infty}$.
Thus, for every $\eps>0$ we can choose $m$ large enough so that $$d_1(\E\mu_{\tau_m\cc_\infty^\sigma},\E\mu_{\cc^\sigma_\infty})<\eps/3$$
and so that $\delta_{\eps,\tau_m}>0$ in Lemma \ref{LDP lemma}.

Eq. \eqref{qi's} and Propositions \ref{lem:moments} and \ref{prop:moment bound} show that the expected LSD for  $\(\tau_m\cc_n\)_{n\in\N}$ exists. Moreover, by Proposition \ref{lemma:same}, it is equal to $\E\mu_{\tau_m\cc_\infty^\sigma}$.
So we may choose $n_0$ large enough so that $n>n_0$ implies
$$d_1(\E\mu_{\tau_m\cc_n},\E\mu_{\tau_m\cc_\infty^\sigma})<\eps/3.$$

Lemma \ref{LDP lemma} and \eqref{lidskii}, show that we may finally choose $n_1$ large enough so that $n>n_1$ implies
$$d_1(\E\mu_{\cc_n},\E\mu_{\tau_m\cc_n})<\eps/3.$$

Combining the above, we have for all $n>\max(n_0,n_1)$,
$$d_1(\E\mu_{\cc_n},\E\mu_{\cc^\sigma_\infty})<\eps$$
and so the ESDs of
$\(\cc_n\)_{n\in\N}$ converge weakly in expectation (and thus a.s.) to $\E\mu_{\cc^\sigma_\infty}$ which is the expected spectral measure at $e_\emp$ of $\cc_\infty^\sigma$ associated to a PWIST($\sigma,\lambda_\Pi$).

The claim that $\mu_{\cc_\infty}$ has bounded support if and only if $\Pi$ is trivial, follows from the remark at the very end of Section \ref{sec:moment method}.
\end{proof}

\begin{proof}[Proof of Corollary \ref{cor2}]
The corollary follows from Theorem 2.1 in \cite{bordenave2011nonHermitian} since it is enough to show the existence of a limiting singular value distribution.  We give a brief outline here and refer the reader to \cite{bordenave2011nonHermitian}
for more details.
Let $\{\sigma_j\}_{j=1}^{n}$ denote the
singular values of the ${n}$th matrix in the sequence $(\aa_n)_{n\in\N}$ and define the symmetrized empirical measure
$$\sigma_{\aa_n}:=\frac{1}{2n}\sum_{j=1}^n(\delta_{\sigma_j}+\delta_{-\sigma_j}).$$
The idea is to associate a $2n\times 2n$ matrix $\BB_{n}$ to each $\aa_n$ by thinking of $\BB_n$ as an $n\times n$ matrix with entries
given by the $2\times 2$ matrices
$$\BB_{n}(j,k):=\, \begin{bmatrix}  0 & \aa_{n}(j,k) \\ \bar\aa_n(j,k) & 0
\end{bmatrix}.$$
Through a permutation of entries, $\BB_n$ is similar to the block matrix
$$\begin{bmatrix}  0 & \aa_{n} \\ \bar\aa_n^* & 0
\end{bmatrix}$$
whose eigenvalues are $\pm\sigma_k(\aa_{n})$. Thus the ESD of $\BB_n$ is precisely equal to $\sigma_{\aa_n}$, and we know that the LSD of $(\BB_n)_{n\in\N}$
exists by Theorem \ref{thm:main}.
\end{proof}

\begin{proof}[Proof of Proposition \ref{cor1}]
The proof is an application of the resolvent identity.  For details, we refer the reader to Proposition 2.1 in \cite{klein1998extended} or Theorem
4.1 in \cite{bordenave2011spectrum}. The latter proof works in our setting almost word for word.
\end{proof}

\appendix
\section{Some additional tools}
\paragraph{Infinite divisibility.}
 The following important set of criteria for convergence to an
infinitely divisible law with characteristics $(\sigma^2,b,\Pi)$ was
found independently by Doeblin and Gnedenko (see Corollary 15.16 in
\cite{kallenberg2002foundations}). For $0<h<1$, define
$$\sigma_h^2:=\sigma^2+\int_{|{{} x}|\le h}{{{} x}^2}\, \Pi(d{{} x})\quad\text{and}\quad b_h:=b-\int_{h<|x|} {\frac{x}{1+x^2}}\,\Pi(d{{} x}).$$
Also, let $\overline{\R }$ be the one-point compactification of
$\R $.
\begin{prop}[Convergence criteria for triangular
arrays]\label{Kall cor15.16} Suppose
$\{\cc(n,k),1\le k\le n\}_{n\in\N}$ is a triangular
array of random variables such that each row consists of
i.i.d. random variables. The sum
$$\sum_{j=1}^n \cc(n,j)$$
converges in distribution to an $ID(\sigma^2,b,\Pi)$ random variable  if and only if for
any $0<h<1$ which is not an atom of $\,\Pi$,
\begin{itemize}
\item $n\P(\cc(n,1)\in \cdot){\vague} \Pi$\ on \ $\overline{\R }\backslash\{0\}$,
\item $n\E\( |\cc(n,1)|^21_{\{|\cc(n,1)|\le h\}}\)\to\sigma^2_h$,
\item $n\E\( \cc(n,1)1_{\{|\cc(n,1)|\le h\}}\)\to b_h$.
\end{itemize}
\end{prop}

%
%
\paragraph{From the \CS\ transform to LSDs.}
The use of the \CS\ transform
in the context of random matrices dates back to Mar\v{c}enko and Pastur
\cite{marchenko1967distribution}.
Mainly, one obtains convergence of
the ESDs of  the random matrices $\(\cc_n\)_{n\in\N}$ by
showing convergence of the \CS\ transforms $\(S_{\mu_{\cc_n}}(z)\)_{n\in\N}$ as defined in \eqref{def:Stieltjes}.
The lemma given here is taken from Section 2.4 in \cite{anderson2010introduction}.

The \CS\ transform is invertible: For any open interval $I$ such that neither
endpoint is an atom of $\mu$
\be\label{st inversion}
\mu(I)=\lim_{y\to 0}\frac 1 \pi \int_I \Im S_\mu(x+iy) \, dx.
\ee
This uniquely determines the
measure $\mu$ so that one then obtains the following result:

\begin{lem}[{Weak} convergence via \CS\ transforms]\label{lem:2.4.4}
Suppose $\mu_{n}$ is a sequence of probability measures on $\R$
and for each $z\in \C_+$, $S_{\mu_{n}} (z)$ converges to $S(z)$ which is the \CS\
transform of some probability measure $\mu$. Then $\mu_{n}$ converges
weakly to $\mu$.
\end{lem}

\begin{proof}
Let $n_k$ be a subsequence for which $\mu_{n_k}$ converges vaguely
to some sub-probability measure $\mu$. For every $z\in\C_+$,
$x\mapsto \frac 1 {x-z}$ is continuous and goes to $0$ as
$x\to\infty$.  Thus one has $S_{\mu_{n_k}}(z)\to S_\mu(z)$ pointwise
for each $z\in\C_+$. By the hypothesis, we have $S(z)=S_\mu(z)$. We
then use \eqref{st inversion} to see that every subsequence gives us
the same limit which implies that $\mu_n$ converges vaguely to
$\mu$. But $\mu$ is a probability measure {by hypothesis}, thus we upgrade this to
weak convergence.
\end{proof}

\newcommand{\etalchar}[1]{$^{#1}$}

\end{document}